\documentclass[12pt]{amsart}
\usepackage{stmaryrd}
\usepackage{mathrsfs}
\usepackage{amssymb}

\usepackage{titletoc}
\input xy
  \xyoption{all}
\usepackage{amscd}
\usepackage{amsmath, amssymb}
\usepackage{amsfonts}
\usepackage[colorlinks,linkcolor=blue,citecolor=blue, pdfstartview=FitH]{hyperref}

\numberwithin{equation}{section}

\theoremstyle{plain}
\newtheorem{thm}{Theorem}[section]
\newtheorem{theorem}[thm]{Theorem}
\newtheorem{lemma}[thm]{Lemma}
\newtheorem{corollary}[thm]{Corollary}
\newtheorem{proposition}[thm]{Proposition}

\theoremstyle{definition}

\newtheorem{remark}[thm]{Remark}

\newtheorem{definition}[thm]{Definition}

\newtheorem{example}[thm]{Example}

\newtheorem{defn-thm}[thm]{Definition-Theorem}

\newcommand{\G}{{\mathbb G}}

\renewcommand{\P}{{\mathbb P}}

\newcommand{\Z}{{\mathbb Z}}

\newcommand{\btheorem}{\begin{theorem}}
\newcommand{\etheorem}{\end{theorem}}
\newcommand{\bproposition}{\begin{proposition}}
\newcommand{\eproposition}{\end{proposition}}
\newcommand{\bdefinition}{\begin{definition}}
\newcommand{\edefinition}{\end{definition}}
\newcommand{\bcorollary}{\begin{corollary}}
\newcommand{\ecorollary}{\end{corollary}}
\newcommand{\bproof}{\begin{proof}}
\newcommand{\eproof}{\end{proof}}
\newcommand{\bremark}{\begin{remark}}
\newcommand{\eremark}{\end{remark}}
\newcommand{\eexample}{\end{example}}
\newcommand{\bexample}{\begin{example}}

\newcommand{\elemma}{\end{lemma}}
\newcommand{\blemma}{\begin{lemma}}

\renewcommand{\bar}{\overline}

\renewcommand{\phi}{\varphi}

\newcommand{\ee}{\end{eqnarray*}}
\newcommand{\be}{\begin{eqnarray*}}

\newcommand{\beq}{\begin{equation}}
\newcommand{\eeq}{\end{equation}}

\newcommand{\bd}{\begin{enumerate}}
\newcommand{\ed}{\end{enumerate}}

\renewcommand{\tilde}{\widetilde}

\renewcommand{\>}{\rightarrow}

\usepackage{fancyhdr}
\begin{document}
\title{On certain K-equivalent birational maps}
\makeatletter
\let\uppercasenonmath\@gobble
\let\MakeUppercase\relax
\let\scshape\relax
\makeatother
\author{ Duo Li}

\address{{Address of Duo Li: Yau Mathematical Sciences Center, Tsinghua University, Beijing, 100084, P. R. China.}}
\email{\href{mailto:liduo211@mails.ucas.ac.cn}{{liduo211@mails.ucas.ac.cn}}}

\maketitle

\begin{abstract}We study K-equivalent birational maps which are resolved by a single blowup.  Examples of such maps include standard flops and twisted Mukai flops. We give a criterion for such maps to be a standard flop or a twisted Mukai flop. As an application, we classify all such birational maps up to dimension 5.

\end{abstract}

\setcounter{tocdepth}{1} \tableofcontents

\section{Introduction} \noindent\\

In this article, all the varieties are defined over complex numbers and are assumed to be projective and smooth. By projective bundle, we mean that there  exists a locally free sheaf  $\mathcal E$ s.t. $\P(\mathcal E)=Proj ( Sym (\mathcal E^{\vee})).$ By $\P^n$ fibration $\alpha:X\>Y,$ we mean that $\alpha^{-1}(y)$ is isomorphic to $\P^n$ for every closed point $y\in Y.$  \\

In this article, we address the study of the birational map $\xymatrix{\theta: X\ar@{-->}[r]&X^+}$  which is K-equivalent and can be resolved by a single blowup.  We call $\theta$ a birational map of \textbf{Simple Type}.  To be concrete, there is a closed smooth subvariety $P$ of $X$ (resp. $P^+$  a closed smooth subvariety of $X^+$) and the blowup  $\varphi: Bl_P(X)\> X$ along $P$ (resp. $\varphi^+: Bl_P^+(X^+)\> X^+$ along $P^+$ ) satisfies that:

\bd
\item $Bl_P(X)\simeq Bl_{P^+}(X^+)$ (we denote this variety by $\tilde X$).
 \item $\theta \circ \varphi=\varphi^+,$ i.e., the diagram $$\xymatrix{&\tilde X\ar[dl]_{\varphi}\ar[dr]^{\varphi^+}&\\X\ar@{-->}[rr]^{\theta}&&X^+}$$ is commutative.
     \item $\varphi^*(K_X)=\varphi^{+*}(K_{X^+})$ in $Pic(\tilde X).$
     \item $\theta$ is not an isomorphism.\ed

 Actually, we know two concrete constructions of birational maps of simple type as follows:\\

$1$. standard flop (over a base $C$)(In article \cite{W.C.L}, standard flops are called ordinary flops instead): \\

Suppose that there is a closed subvariety $P$ of $X$ satisfies that  $\pi:P=\P_C(\mathcal E)\>C$ is a projective bundle. The normal bundle of $P$ satisfies that $N_{P/X}=\pi^*\mathcal E^+\otimes \mathcal O(-1)$ where $\mathcal E^+$ is a vector bundle over $C$ and $\mathcal O(1)$ is the tautological bundle for $\P_C(\mathcal E)$. Then we consider the blowup of $X$ along $P.$ Let $P^+=\P_C(\mathcal E^+).$ We denote $\pi^+: P^+\>C$ the projective bundle morphism,  then the  exceptional divisor $E\simeq P\times_C P^+.$ For any fibre $F$ of $p^+: E\>P^+,$ by a simple calculation, we have that $\mathcal O_E(E)|_F\simeq \mathcal O(-1)$ (see Chapter 11, \cite{DH} or Section 1, \cite{W.C.L}). By Fujiki-Nakano's criterion(see \cite{F} or Remark 11.10, \cite{DH}), there is a blow-down $\varphi^+: \tilde{X}\>X^+$ compatible with $p^+.$ Hence we obtain a birational map $\theta:X \dashrightarrow X^+,$ we call $\theta$ a standard flop over $C.$\\

$2$. twisted Mukai flop (over a base $C$):\\

Suppose that there is a closed subvariety $P$ of $X$ satisfies that  $\pi:P=\P_C(\mathcal E)\>C$ is a projective bundle. The normal bundle of $P$ satisfies that $N_{P/X}=\Omega_{P/C}\otimes \pi^*\mathcal L$ for some line bundle $\mathcal L $ of $C.$  Then we consider the blowup of $X$ along $P.$ Let $P^+=\P_C(\mathcal E^{\vee}).$ We denote $\pi^+: P^+\>C$ the projective bundle morphism,  then the  exceptional divisor $E\simeq \P_P(\Omega_{P/C})$ is a prime divisor of $P\times_C P^+,$ hence there is an induced projective bundle structure $p^+: E\>P^+.$ For any fibre $F$ of $p^+,$ by a simple calculation, we have that $\mathcal O_E(E)|_F\simeq \mathcal O(-1)$ (see Chapter 11, \cite{DH} or Section 6 of \cite{W.C.L}). By Fujiki-Nakano's criterion(see \cite{F} or Remark 11.10, \cite{DH}), there is a blow-down $\varphi^+: \tilde{X}\dashrightarrow X^+$ compatible with $p^+.$ Hence we obtain a birational map $\theta:X \dashrightarrow X^+,$ we call $\theta$ a twisted Mukai flop over $C,$ if $\mathcal L$ is a trivial line bundle, we call $\theta$ a Mukai flop for short.\\
\\

\bremark\label{W.C} In the above constructions, $X^+$ is not necessarily projective. But it is reasonable to assume that there exists a flopping contraction $\beta:X\>\bar{X}$ compatible with $\pi:P\>C.$  Under this assumption, $X^+$ can be proved projective. For details, see Proposition 1.3 and Proposition 6.1, \cite{W.C.L}.\eremark There is a natural question: are all birational maps of simple type either  standard flops or  twisted Mukai flops? Actually, we construct a new simple type birational map in Example \ref{Example} and one of the  main results  of our article is:
\btheorem A birational map of simple type, if $\dim X\le 5,$
is either a  standard flop, a twisted Mukai flop or a flop as in Example \ref{Example}.
\etheorem
For properties about  standard flops and twisted Mukai flops, we use \cite{DH} and \cite{W.C.L} as our main references. We note that in articles \cite{BW} and \cite{W.C.L}, motivic and quantum invariance under  a standard flop or a twisted Mukai flop was studied, we hope this article could offer useful information to further study in that direction.
\\

Now let us state the structure of this article:  first, we  prove that $\varphi$ and $\varphi^+$ share a common exceptional divisor $E$ and $\dim P= \dim P^+.$  Then we generalise  E. Sato' classification results about varieties which admits two different projective bundle structures  to a relative version. The main result of this article is Theorem \ref{main result}, we prove that a birational map of simple type is a standard flop or a twisted Mukai flop if and only if $P$ and $P^+$ are projective bundles over a common variety. As an application of our main result, we classify simple type birational maps when $\dim X\le 5,$ here  recent results of \cite{K} and \cite{W} play important roles in our classification.\\

\textbf{Acknowledgments.} The author is  very grateful to Professor Baohua
Fu for his support, encouragement and stimulating discussions over
the last few years. The author is very grateful to Professor Chin-Lung Wang for his helpful suggestions and discussions. The author wishes to thank Yang Cao, Yi Gu, Wenhao Ou,  Xuanyu Pan, Lei Zhang  for useful discussions and thank Professor Xiaokui Yang for his support and encouragement.
 \section {Proof of main results}Let us start with the following observation, which shows that $\varphi$ and $\varphi^+$ share a common exceptional divisor $E$, hence $E$ admits two projective bundle structures. We will repeatedly use this fact in the rest of our article.
 \begin{lemma}The blowups $\varphi$ and $\varphi^+$ share a common exceptional divisor $E$ and $\dim P=\dim P^+$\end{lemma}
 \begin{proof}Let $k=codim(P,X)$ and  $k^+=codim(P^+,X^+),$ then $k,$ $k^+\ge2.$ Since $\varphi$ and $\varphi^+$ are blowups, the corresponding exceptional divisors are $E$ and $E^+,$  we have the following two equations: \bd\item $\varphi^*K_X=K_{\tilde X}-(k-1)E$\item $\varphi^{+*}K_{X^+}=K_{\tilde X}-(k^+-1)E^+.$\ed Since $\varphi^*(K_X)=\varphi^{+*}(K_{X^+}),$ $E$ is numerically equivalent to $\frac {k^+-1}{k-1}E^+,$ where $\frac {k^+-1}{k-1}>0.$\\

Since $E$ is a projective bundle over $P$, for any point $p\in P$, the fiber $F_p$ of $\varphi$ is isomorphic to $\mathbb P^{k-1}$. For any line $C$ in $F_p\simeq \mathbb P^{k-1}$, $C\cdot E=\deg\mathcal  O_{\tilde X}(E)|_C=\deg\mathcal  O(-1)=-1.$ As $p$ varies, these negative lines in  $F_p$ cover $E.$\\

  If $E\ne E^+, $ then there exists a line $C$ in  $F_{p_0}$ for some $p_0\in P $ such that $C$ is not contained in $E^+.$ So $C\cdot E=C\cdot (\frac {k^+-1}{k-1})E^+\ge 0, $ contradicts to $C\cdot E<0.$\\

  Then $E=E^+$ and $k^+=k.$ \end{proof}
   By  Propostion 1.14, \cite{De}, $\varphi|_E$ and $\varphi^+|_E$ are non-isomorphic fibrations. There is a lower bound for the dimension of $P.$
  \begin{lemma}\label{lower bound}$2\dim P+1\ge \dim X.$\end{lemma}
  \begin{proof} The exceptional divisor $E$ has two projective bundle structures over $P$ and $P^+$. We denote these two projective bundle morphisms in the following diagram:
  $$\xymatrix{E\ar[r]^p\ar[d]_{p^+}&P\\P^+}$$ where $p=\varphi|_E$ and $p^+=\varphi^+|_E.$  Since fibers of different extremal ray contractions can meet only in points,  $(p,p^+):E\longrightarrow P\times P^+$ is a finite morphism to its image. So $\dim E\le 2 \dim P,$ which means that $\dim X\le 2 \dim P+1.$\end{proof}

   Now recall
  E.Sato's  classification  about varieties which admit two projective bundle structures (See Theorem A in ~\cite{Sato}).
  \begin{theorem}\label{E.Sato}
  If $E$ has two projective bundle structures over projective spaces:  $$\xymatrix{E \ar[r]^{p_1} \ar[d]_{p_2} & \mathbb P^n\\  \mathbb P^n}$$ then $E$ is isomorphic to  either
  $\mathbb P^n\times \mathbb P^n$ ( in this case,  $\dim E=2n$ ) or    $\mathbb P_{\mathbb P^n}(\Omega _{\mathbb P^n})$ (in this case, $\dim E=2n-1$).
  \end{theorem}

 \bremark\label{sketch} We now sketch the proof of E.Sato's theorem in the case of $\dim E=2n-1$.  First, it can be proved that  $\Phi=(p_1,p_2): E \longrightarrow \mathbb P^n\times \mathbb P^n$ is  a closed immersion. Then $\Phi(E)$ is a Cartier divisor in $\mathbb P^n\times \mathbb P^n,$  we denote the defining equation of $\Phi(E)$ by $F(X_0，\cdots, X_n; Y_0，\cdots, Y_n).$ The key point is  $F$ is a homogeneous polynomial of bidegree $(1,1).$  Moreover,  it can be shown that after performing  suitable linear transforms  of $(X_0，\cdots, X_n)$ and $(Y_0，\cdots, Y_n),$ $F$ is of the form $F=\sum_{i=0}^{n}{X_i\cdot Y_i}.$  If we view $ \mathbb P_{\mathbb P^n}(\Omega _{\mathbb P^n})$ as a closed subvariety of $\P^n\times (\P^n)^*$ by the Euler sequence:
 $$0\>\Omega_{\mathbb P^n}\>\mathcal{O}(-1)^{n+1}\> \mathcal{O}\>0,$$ then $E$ has the same defining equation as  that of $ \mathbb P_{\mathbb P^n}(\Omega _{\mathbb P^n}).$\eremark

  Now we generalise E.Sato's result to a relative version.

  \begin{theorem}\label{main theorem}  Suppose that $E$ admits two projective bundle structures:$$\xymatrix{E \ar[r]^{p_1} \ar[d]_{p_2} & P_1\\  P_2}$$ If the following conditions are satisfied:
  \bd
  \item $P_1$ and $P_2$ are projective bundles over a common variety $C,$ i.e.  $P_1=\mathbb P_{C}(\mathcal E_1)$ and $P_2=\mathbb P_{C}(\mathcal E_2)$ for some locally free sheaves  $\mathcal E_i$ on $C$. \item Let $\pi_i: P_i \longrightarrow C$ denote the projective bundle morphism, we assume that $\pi_1\circ p_1=\pi_2\circ p_2.$
  \item $\dim P_1=\dim P_2.$

  \ed then $E$ is isomorphic to either $P_1\times_C P_2$ or
 $\mathbb P_{ P_1}(\Omega _{P_1/C}).$   In the first case, $\dim E=2\dim P_1-\dim C;$ in the second case, $\dim E =2\dim P_1-\dim C-1.$

  \end{theorem}
  \begin{proof} We assume that  the rank of $\mathcal E_i$ is $k+1$ and let $\alpha=\pi_1\circ p_1.$ Then the following diagram :
     $$\xymatrix{E\ar[rd]^{\alpha}\ar[r]^{p_1}\ar[d]_{p_2}&P_1\ar[d]^{\pi_1}\\P_2\ar[r]_{\pi_2}&C}$$ commutes. Since every fiber of $\alpha$ has two projective bundle structures over projective spaces, by Theorem \ref{E.Sato}, $\dim E=2\dim P_1-\dim C$ or $\dim E =2\dim P_1-\dim C-1.$ \\

   In the case of  $\dim E=2\dim P_1-\dim C,$ we aim to show that $E\simeq P_1\times_C P_2. $\\ Note that we have the following commutative diagram:  $$\xymatrix{ E\ar[rr]^{\Phi=(p_1,p_2)}\ar[rd]_{\alpha}& &P_1\times_C P_2\ar[ld]\\&C}$$ \\ For any $c\in C,$   let $E_c$ be the fiber of $\alpha.$  Then by Theorem \ref{E.Sato}, the restriction of $\Phi$ to $E_c:$ $$\Phi_c: E_c\longrightarrow \pi_1^{-1}(c)\times\pi_2^{-1}(c)\simeq \mathbb P^k\times \mathbb P^k $$ is an isomorphism. So $\Phi$ is an isomorphism.\\

 In the case of $\dim E=2\dim P_1-\dim C-1,$ we aim to show that $E\simeq \mathbb P_{ P_1}(\Omega _{P_1/C}).$\\

  We note that $\Phi=(p_1,p_2): E\longrightarrow P_1\times_C P_2$ is a closed immersion. First, $\Phi$ is finite onto its image. Otherwise, there will be a curve contracted by both $p_i.$ Since for any curve $C$ of $\P^k,$ deformations of $C$ cover $\P^k,$ then there is a fibre of $p_1$ contracted by $p_2.$ So by the rigidity lemma, $p_1=p_2$ as fibrations, which is a contradiction to our assumptions. Also, for an arbitrary $c\in C,$ by Theorem \ref{E.Sato},  $\Phi_c: E_c\longrightarrow \pi_1^{-1}(c)\times\pi_2^{-1}(c)\simeq \mathbb P^k\times \mathbb P^k $ is a closed immersion. Moreover, $E_c$ is a prime divisor defined by a homogeneous polynomial  of bidegree $(1,1).$ Then $\Phi$ is a finite morphism to its image of degree one, so $\Phi$ is a closed immersion. \\

  Let $X=P_1\times_C P_2,$   $q_i: X\>P_i$ denote the projection morphism and $f:X\>C$ be the base morphism.  We view $E$ as a prime divisor of $X,$  then $\mathcal O_X(E) \simeq q_1^*\mathcal O(1)\otimes q_2^*\mathcal O(1)\otimes f^*\mathcal L$ where $\mathcal O(1)$ denotes the tautological line bundle of a given projective bundle and $\mathcal L$ is a line bundle of $C.$ Note that $f_*(\mathcal O_X(E))\simeq \mathcal E_1^{\vee}\otimes \mathcal E_2^{\vee}\otimes \mathcal L.$ Since $E$ is an effective divisor, there is a non-zero global section of $H^0(C,f_*(O_X(E)))=H^0(X, \mathcal O_X(E) )$. Then this global section induces a morphism of sheaves $\mathcal O\>\mathcal E_1^{\vee}\otimes \mathcal E_2^{\vee}\otimes \mathcal L,$  hence a linear form $\beta: \mathcal E_1\otimes \mathcal E_2 \>\mathcal L.$ \\

  Actually, the linear form $\beta$ is non-degenerate. For any affine open subset $Spec$$ A$ of $C, $ $\Phi(E) $ is a Cartier divisor of $P_1\times_A P_2.$ Then $\Phi(E)$ is defined by a single equation $F(X_0，\cdots, X_k; Y_0，\cdots, Y_k)=0$ whose coefficients are elements of $A.$  For any local sections of $\mathcal E_1$ and $\mathcal E_2$: $e_1=(x_0，\cdots, x_k)$,  $e_2=(y_0，\cdots, y_k)$, we have that $\beta(e_1,e_2)=F(x_0，\cdots, x_k; y_0，\cdots, y_k).$ Since $F$ is a homogeneous polynomial of bidegree $(1,1)$,  we can write $\beta$ as :\begin{equation}
  \beta(e_1,e_2)=
  \left(
  \begin{array}{ccc}
          x_{0} &
          \cdots &
          x_{k}
  \end{array}
  \right)
  M
  \left(
  \begin{array}{c}
          y_{0} \\
          \vdots \\
          y_{k}
 \end{array}
 \right)
\end{equation} where $M$ is a matrix in $M_n(A).$
By Theorem \ref{E.Sato} and Remark \ref{sketch}, $\det(M)$ is invertible in $A_m$ for any maximal ideal $m$ of $A,$ so $M$ is an invertible matrix, which means that the linear form $\beta(-,-)$is non-degenerate.  So $\mathcal E_1\simeq \mathcal E_2^{\vee}\otimes \mathcal L.$ Since $\P(\mathcal E_2\otimes \mathcal L^{-1})\simeq \P(\mathcal E_2),$ in what follows we assume that $\mathcal E_1\simeq \mathcal E_2^{\vee}.$

By the relative Euler sequence $0\>\Omega _{P_1/C}\>\pi_1^*(\mathcal E_1^{\vee})(-1)\>\mathcal O\>0,$ $\mathbb P_{ P_1}(\Omega _{P_1/C})$ is a closed subvariety of $P_1\times_C P_2.$ As closed subvarieties of $P_1\times_C P_2,$ on any affine piece of $C,$ $E$ and  $\mathbb P_{ P_1}(\Omega _{P_1/C})$ are defined by the same equation $\sum X_i\cdot Y_i=0,$  so    $E\simeq \mathbb P_{ P_1}(\Omega _{P_1/C}).$   \end{proof}

Now we assume that $P$ and $P^+$ are projective bundles over a common variety $C$, i.e. $P\simeq \mathbb P_{C}(\mathcal E) $ and $P^+\simeq \mathbb P_{C}(\mathcal E^+)$ where $\mathcal E$, $\mathcal E^+$ are locally free sheaves of rank $n+1$ on $C$, $\pi:P=\mathbb P_{C}(\mathcal E)\longrightarrow C$ (resp. $\pi^+:P^+=\mathbb P_{C}(\mathcal E^+)\longrightarrow C$) is the projective bundle morphism.
     We let $p=\varphi|_E$ and $p^+=\varphi^+|_E.$

    We summarise all the morphisms of our problem in the following diagram:
     $$
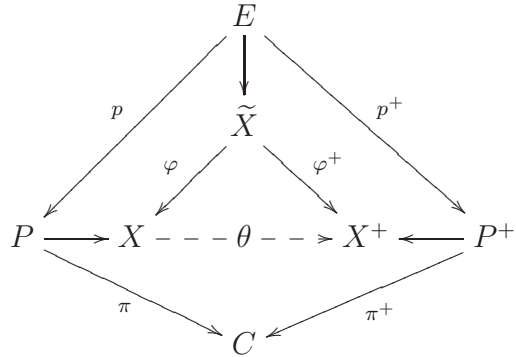
\xymatrix{& &E\ar[d]\ar[ddrr]^{p^+}\ar[ddll]_{p}\\ & &\tilde X\ar[dl]_{\varphi}\ar[dr]^{\varphi^+}\\P\ar[r]\ar[rrd]_{\pi}&X\ar@{--}[r]&\theta \ar@{-->}[r] &X^{+}&P^{+}\ar[l]\ar[lld]^{\pi^+}\\&&C&&}.$$
     \begin{theorem}\label{main result} Suppose that $\theta: (X,P)\dasharrow (X^+,P^+)$ is a birational map of simple type, where $(P,\pi)$ and $(P^+,\pi^+) $ are projective bundles over a common variety $C$ and $\pi\circ p=\pi^+\circ p^+.$ Then $\theta$ is a  standard flop or a twisted Mukai flop over base $C.$\end{theorem}

     \begin{proof}Since $E$ has two projective bundle structures over $P$ and $P^+$, by Theorem \ref{main theorem}, $E\simeq P\times_C P^+$ or $E\simeq\mathbb P_{ P}(\Omega _{P/C}).$\\

    Assume first that $E\simeq P\times_C P^+.$
     Since $E \simeq   \mathbb P_{ P}(\pi^*\mathcal E^+)$ and $E\simeq \mathbb P_{ P}( N_{P/X}) ,$  we can assume that $ N_{P/X}=\pi^*\mathcal E^+\otimes \mathcal L$ where  $\mathcal L$ is an invertible sheaf of $P.$  Suppose that  $\mathcal L=\pi^*\mathcal L'\otimes \mathcal O (-d),$ where $\mathcal O(1)$ is the tautological line bundle of $\pi:P\>C$ and $\mathcal L'$ is a line bundle of $C.$ Since $\mathbb P_{C}(\mathcal E^+)=\mathbb P_{C}(\mathcal E^+\otimes \mathcal L'),$  we can assume that $ N_{P/X}=\pi^*\mathcal E^+\otimes\mathcal O (-d). $  We aim to show that $d=1.$ \\

     Since $\varphi$ is a blowup, $\omega_E\simeq \varphi^*\omega_X|_E\otimes \mathcal O_E((n+1)E)$ where $n+1=codim(P,X).$
Let $h=p \circ \pi.$ We now calculate $\varphi^*\omega_X|_E\simeq p^*\omega_X|_P$ as follows: \begin{equation} \label{E}
       \varphi^*\omega_X|_E  \simeq p^*\omega_P\otimes(\det N_{P/X})^{-1}\simeq p^*\omega_P\otimes h^*(\det\mathcal{E^+})^{-1}\otimes p^*\mathcal{O}((n+1)d).
     \end{equation}
So $\omega_E\simeq p^*\omega_P\otimes h^*(\det\mathcal{E^+})^{-1}\otimes p^*\mathcal{O}((n+1)d)\otimes \mathcal O_E((n+1)E).$
      For any fiber $F\simeq \P^n$ of $p^+,$   $h(F)=c$ for some point $c\in C.$ The morphism  $p$ maps $F$ isomorphically to $\pi^{-1}(c)=\P^n,$ which is illustrated as follows:$$\xymatrix{
          F \ar[d]_{p|_F} \ar[r]^{} & E \ar[rd]^h \ar[d]_{p} \ar[r]^{p^+} & P^+ \ar[d]^{\pi^+} \\
           \pi^{-1}(c)\ar[r]^{} & P \ar[r]^{\pi} & C   }$$
           Since $\mathcal O_E(E)|_F\simeq \mathcal O(-1),$ then we have that $$\{p^*\omega_P\otimes h^*(\det\mathcal{E^+})^{-1}\otimes p^*\mathcal{O}((n+1)d)\otimes \mathcal O_E((n+1)E)\}|_F\simeq \mathcal O((n+1)(d-2)).$$ We know that $\omega_E|_F\simeq \mathcal{O}(-n-1),$ so $d=1$ which means that $N_{P/X}\simeq \pi^*\mathcal E^+\otimes \mathcal O(-1)$ and $\theta$ is a standard flop over $C.$\\

Assume now that $E\simeq \mathbb P_{ P}(\Omega _{P/C}).$ By Theorem \ref{main theorem}, we know that $\mathcal E\simeq \mathcal E^{+\vee}$ and $E\hookrightarrow P\times_C P^+$ is a closed immersion.
     Since $E\simeq \mathbb P_{ P}( N_{P/X}) ,$  we can assume that $ N_{P/X}=\Omega_{p/C}\otimes \mathcal L$ where  $\mathcal L$ is an invertible sheaf of $P.$  Suppose that  $\mathcal L=\pi^*\mathcal L'\otimes \mathcal O (-d),$ where $\mathcal O(1)$ is the tautological line bundle of $\pi:P\>C$ and $\mathcal L'$ is a line bundle of $C.$   We aim to show that $d=0.$ \\

     Since $\varphi$ is a blowup, $\omega_E\simeq \varphi^*\omega_X|_E\otimes \mathcal O_E((n+1)E)$ where $n+1=codim(P,X).$
Let $h=p \circ \pi,$ we now calculate $\varphi^*\omega_X|_E\simeq p^*\omega_X|_P$ as follows: \begin{equation} \label{F}
       \varphi^*\omega_X|_E  \simeq p^*\omega_P\otimes(\det N_{P/X})^{-1}\simeq  h^*\omega_C\otimes p^*\mathcal{L}^{\otimes(-n-1)}.
     \end{equation}
So $\omega_E\simeq h^*\omega_C\otimes p^*\mathcal{L}^{\otimes(-n-1)}\otimes \mathcal O_E((n+1)E).$
      For any fiber $F\simeq \P^n$ of $p^+,$   $h(F)=c$ for some point $c\in C.$ The morphism  $p$ maps $F$  into $\pi^{-1}(c)=\P^{n+1}$ as a hyperplane,  which is illustrated as follows:$$\xymatrix{
          F \ar[d]_{p|_F} \ar[r]^{} & E \ar[rd]^h \ar[d]_{p} \ar[r]^{p^+} & P^+ \ar[d]^{\pi^+} \\
           \pi^{-1}(c)\ar[r]^{} & P \ar[r]^{\pi} & C   }$$
           Since $\mathcal O_E(E)|_F\simeq \mathcal O(-1),$ then we have that $$\{ h^*\omega_C\otimes p^*\mathcal{L}^{\otimes(-n-1)}\otimes \mathcal O_E((n+1)E)\}|_F\simeq \mathcal O((n+1)(d-1)).$$ We know that $\omega_E|_F\simeq \mathcal{O}(-n-1),$ so $d=0$ which means that $N_{P/X}\simeq \Omega_{P/C}\otimes \pi^*\mathcal L'$ and $\theta$ is a twisted Mukai flop over $C.$

      \end{proof}

     Note that by Lemma \ref{lower bound}, we know that there is a lower bound for the dimension of $P.$ As an application of Theorem \ref{main result}, we will classify the birational morphism $\theta$ when $\dim P $  reaches the lower bound. Here, what is different from Theorem \ref{main result}, we don't assume that $P$ and $P^+$ are projective bundles in advance, actually we obtain a result as follows:

     \begin{theorem}\label{2}
       \bd
         \item If $\dim X=2\dim P+1,$ then $\theta$ is a standard flop.
         \item If $\dim X=2\dim P,$ then $\theta$ is a Mukai flop or a  standard flop over a curve.
       \ed
     \end{theorem}
     \begin{proof} \begin{enumerate}\item If $\dim X=2\dim P+1,$ then $\dim E=2\dim P.$ So the morphism $(p,p^+):E\longrightarrow P\times P^+$ is surjective. For any fiber $F$ of $p^+$, $p|_{F}:F\longrightarrow P$ is surjective. By Lazarsfeld's theorem ~\cite{L}, $P$ is a projective space, similarly, $P^+$ is also a projective space. By Theorem \ref{main result}, $\phi$ is a standard flop.
     \item If $\dim X=2\dim P,$ then $\dim E=2\dim P-1.$ By Theorem 2 in ~\cite{A}, $E\simeq \mathbb P_{ P}(\Omega _{P})$ where $P$ and $P^+$ are projective spaces or $E\simeq  P\times_C P^+$ where $C$ is a smooth curve and $P$ (resp. $P^+$) is a projective bundle over $C.$ By Theorem \ref{main result}, $\phi$ is a Mukai flop or a standard flop.\end{enumerate}
     \end{proof}
\section{Classifications when $\dim X\le 5$}
     Now as an application of all the results we have obtained, we can classify the simple type birational maps when $\dim X\le5.$ As a first step, we have the following easy corollary from Theorem \ref{2}.
     \bcorollary\label{Corollary}
       \bd
         \item If $\dim X=3$, then by Lemma \ref{lower bound}, we have $2\dim P+1\ge 3$, so $\dim P=1$. Then $\theta$ is a standard flop by Theorem \ref{2}.
         \item If $\dim X=4$, then by Lemma \ref{lower bound}, we have $2\dim P+1\ge 4$, so $\dim P=2$. Then $\theta$ is a Mukai flop or a  standard flop by Theorem \ref{2}.
         \item If $\dim X=5$, then by Lemma \ref{lower bound}, we have $2\dim P+1\ge 5$, so $\dim P=2$ or $3$. If $\dim P=2,$ then $\theta$ is a Mukai flop or a  standard flop by Theorem \ref{2}.

       \ed
     \ecorollary
In the rest of this section, we keep the assumption that  $\dim X=5.$ As we see in Corollary \ref{Corollary}, the remaining unknown case is $\dim P=3.$ In this situation, the exceptional divisor $E$ admits two $\P^1$ bundle structures as follows:
$$\xymatrix{E \ar[r]^{p} \ar[d]_{p^+} & P\\  P^+}.$$ For an arbitrary fibre $F$ of $p^+,$ it is a natural question to ask whether $p(F)$ is extremal in the cone of curves $\bar{NE(P)}.$ Actually, we have the following lemma.
\blemma
If $E$ admits two $\P^1$ bundle structures as above,  then $p(F)$ is extremal. Furthermore, the contraction morphism $\pi: P\>C$ is smooth.\elemma
\bproof See Theorem 2.2, \cite{K}.\eproof
If the picard number $\rho(P)\ge 2,$ i.e. $\dim C\ge 1,$ we aim to show that $\theta$ is a standard flop or a twisted Mukai flop. First, there exists $\pi^+:P^+\>C$ making  the following diagram commutative: $$\xymatrix{E\ar[r]^{p}\ar[d]_{p^+}&P\ar[d]^{\pi}\\P^+\ar[r]_{\pi^+}&C}.$$ The main obstruction to apply our result Theorem
\ref{main result} is that we don't know,  a priori, whether $\pi$ or $\pi^+$ is a projective bundle. Actually, there is a criterion for projective bundles as follows:
\blemma Suppose that  $f:M\>S$ is a smooth $\P^k$ fibration. Let us consider the exact sequence $$0\>\G_m\>GL_{k+1}\>PGL_k\>0,$$ then we have an exact sequence of \'{e}tale cohomologies:$$H^1_{\acute{e}t}(S,GL_{k+1})\stackrel{d}\> H^1_{\acute{e}t}(S,PGL_{k})\>H^2_{\acute{e}t}(S,\G_m),$$ if $d$ is surjective, then $f$ is a projective bundle. In particular, when $\dim S=1,$ $H^2_{\acute{e}t}(S,\G_m)$ vanishes and $d$ is surjective.\elemma
\bproof See Theorem 0.1 in \cite{MM1} and Lemma 1.2 in \cite{MM2}.\eproof

\bremark\label{counterexample} Note that when $\dim Z\ge 2,$ the $H^2_{\acute{e}t}(S,\G_m)$ is not necessarily vanished, things become much more complicated. For example, we consider an arbitrary smooth $\P^1$ fibration $f:M\>S$ which is not a projective bundle. Let $R=M\times_S M,$ the projection $p: R\>M$ is always a projective bundle, as the diagonal morphism $\delta: M\> R$ is a section of $p$ (for details, see Chapter 3, Exercise 4.24 and Chapter 4,  \cite{Milne}).  This example shows that the projective bundle structure of $p$ can't decent to $f$, it is the main difficulty in the classification of simple type birational maps when $\dim X=5.$ \eremark

\blemma If $\dim C=1,$ then $\theta$ is a twisted Mukai  flop over $C.$ \\
If $\dim C=2,$ then $\theta$ is a  standard flop over $C.$
\elemma

\bproof First, we fix some notations. For any $c\in C,$ we let $P_c$ be the fibre of $\pi$ and $P^+_c$ be the fibre of $\pi^+.$ Let $h=\pi\circ p$ and $E_c=h^{-1}(c).$ Since $\pi$ is an elementary contraction, the picard numbers $\rho(P_c)=\rho(P^+_c)=1.$

We observe that the exceptional divisor $E$ admits two $\P^1$ bundle structures over $P$ and $P^+,$ so there are induced $\P^1$ bundle structures of $E_c$ over $P_c$ and $P^+_c.$\\

If $\dim C=1,$ then $\dim E_c=\dim P_c+\dim P^+_c-1.$  By Theorem 2 in ~\cite{A} and $\rho(P)=1,$ $P_z$ and $P^+_z$ are projective spaces. Since $C$ is a curve, $H^2_{\acute{e}t}(C, \G _m)=0,$ then $P$ and $P^+$ are projective bundles over $C.$ Then by Theorem \ref{main result}, $\theta$ is a twisted Mukai flop over $C.$\\

If $\dim C=2,$ then for any $c\in C,$ $\dim P_c=\dim P^+_c=1.$ Then $\dim E_c=\dim P_c+\dim P^+_c,$ by the same argument as in Theorem \ref{2}, we know that $P_c$ and $P^+_c$ are projective lines.  So $\pi$ and $\pi^+$ are smooth $\P^1$ fibrations. However, as we show in Remark \ref{counterexample}, the projective bundle structure of $p$ can't decent to $\pi.$ Fortunately, there is a criterion of projective bundles for elementary contractions, see Lemma \ref{maximal length} below.  Let $l=-P_c\cdot K_P.$ Since $K_P=\pi^*K_C\otimes K_{P/C},$  $l=2.$ Then $\pi$ is an elementary contraction of maximal length. By Lemma \ref{maximal length}, $\pi$ is a projective bundle. By Theorem \ref{main result}, we know that $\theta$ is a standard flop over $C.$  \eproof
\blemma\label{maximal length}Suppose that  $\pi: P\>C$ is an elementary contraction of an extremal ray $\Gamma.$ Let $l(\Gamma)$= max~\{$K_P\cdot A|$ $A$ is a rational curve whose numerical class $[A]$ is in $\Gamma$\}. If $\pi$ is an equidimensional fiber contraction of relative dimension $d$ and $l(\Gamma)=-d-1,$ then $\pi$ is called an elementary contraction of maximal length and $\pi$ is a projective bundle. \elemma
\bproof See Theorem 1.3, \cite{H}. \eproof

Now we assume that $\rho(P)=1$ i.e. $\dim C=0.$ We recall the following result about smooth varieties which admits two $\P^1$ bundle structures:
\blemma \label{classification}Let $X$ be a complex projective manifold with Picard number $\rho(X)=1$ and $\mathcal E$ rank $2$ vector bundle on $X.$ Assume that $Z=\P(\mathcal E)\> X$ admits another smooth morphism $Z\>Y$ of relative dimension $1$ and $n= \dim X\ge2.$ Then,
\bd
\item $X$ and $Y$ are Fano manifolds with $\rho=1$ and there exists a rank $2$ vector bundle $\mathcal E'$ on $Y$ such that $Z\>Y$ is given by $\P(\mathcal E').$
\item if $\mathcal E$ and $\mathcal E'$ are normalised by twisting with line bundles (i.e., $c_1=0$ or $-1$), then $((X,\mathcal E),(Y,\mathcal E'))$ is one of the following, up to exchanging the pairs $(X,\mathcal E)$ and $(Y,\mathcal E')$:
\begin{enumerate}
\item $((\P^2, T_{\P^2}),(\P^2, T_{\P^2})),$ where $T_{\P^2}$ is the tangent bundle of the projective plane $\P^2,$
\item $((\P^3, \mathscr {N}),(Q^3,\mathscr {S}  )),$ where $\mathscr{N}$ is a null-correlation bundle on $\P^3$ and $\mathscr{S}$ is the restriction to the $3-$dimensional quadric $Q^3$ of the universal quotient bundle of the Grassmannian $G(1,\P^3),$

\item $((Q^5,\mathscr{C}),(K(G_2),\mathscr{L})),$ where $\mathscr{C}$ is a Cayley bundle on $Q^5,$ $K(G_2)$ is the $5-$dimensional Fano homogeneous contact manifold of type $G_2$ which is a linear section of the Grassmannian $G(1,\P^6)$ and $\mathscr{L}$ the restriction of the universal quotient bundle on $G(1,\P^6).$
\end{enumerate}
\ed
\elemma
\bproof See Theorem 1.1, \cite{W}.\eproof
Now apply the above Lemma \ref{classification} to our problem, note that $E$ admits two $\P^1$ bundle structures and $\dim E=4,$ we can assume that $P=\P^3$ and the normal bundle $N_{P/X}=\mathscr {N}\otimes \mathcal O (d)$  for some integer $d\in \Z.$ By the exact sequence for the null-correlation bundle $$0\>\mathscr {N}\>T_{\P^3}(-1)\>\mathcal O(1)\>0$$  $E$ is a closed subvariety of $\P^3\times Q^3$ and $p$ is compatible with the projection $\P^3\times Q^3\>\P^3$ (resp. $p^+$ is compatible with the projection $\P^3\times Q^3\>Q^3,$ for details, see Proposition 2.6, \cite{S} and Chapter 1, Section 4.2, \cite{O} ).
\bexample \label{Example}Keep the notations as above, we can construct a simple type birational map as follows:\\

Let $P=\P^3$ and the normal bundle $N_{P/X}=\mathscr{N}\otimes \mathcal O(d).$  We denote  by $\varphi: \tilde{X}\>X$ the blowup of $X$ along $P,$ hence the exceptional divisor $E$ has a $\P^1$ bundle structure over $P.$ By Lemma \ref{classification}, $E$ has another $\P^1$ bundle structure over a quadric $Q^3$, we denote this projective bundle morphism by $p^+: E\> Q^3.$ Then $\omega_E\simeq \varphi^*\omega_X|_E\otimes \mathcal O_E(2E)\simeq p^*\omega_P\otimes(\det N_{P/X})^{-1}\otimes \mathcal O_E(2E).$ \\

Next we calculate the restriction of $\mathcal O_E(E)$ on an arbitrary fiber $F$ of $p^+.$  By  Proposition 2.6, \cite{S}, $F$ is a projective line in $\P^3$ and $c_1(\mathscr{N})=0.$ So $ p^*\omega_P\otimes(\det N_{P/X})^{-1}\otimes \mathcal O_E(2E)|_F=\mathcal O(-4-2d)\otimes \mathcal O_E(2E)|_F.$ Since $\omega_E|_F\simeq \omega_{E/Q^3}|_{F}=\mathcal O (-2).$ So $\mathcal O_E(E)|_F\simeq \mathcal O(d+1).$ By Fujiki-Nakano's criterion(see \cite{F} or Remark 11.10, \cite{DH}), there is a blow-down $\varphi^+: \tilde{X}\>X^+$ compatible with $p^+$ if and only if $d=-2.$\\

Also, $X^+$ constructed here is not necessarily projective. As in Remark \ref{W.C}, suppose that there exists a flopping contraction $\beta:X\>\bar{X}$ contracting  $P,$  then  $X^+$ can be proved  projective by the same argument as Proposition 1.3, \cite{W.C.L}. For the sake of completeness, we here give a detailed proof as follows:
First, $-K_{\tilde X}\cdot F=deg(\omega_E\otimes \mathcal O_E(-E))|_F=-1<0.$ So $F$ is a $K_{\tilde X}$ negative curve. Next, we aim to show that $[F]$ spans an extremal ray in $\overline{NE(\tilde X)},$ i.e., it has a supporting divisor(nef and big). Let $H$ be an ample line bundle on $X$ and $\mathcal L=\beta^*\bar{H},$ where $\bar{H}$ is an ample line bundle on $\bar{X}.$ Then we consider the divisor $$\mathcal L_k'=k \varphi^*\mathcal L-(\varphi^*H+\lambda E)$$ where $k>>0$ and  $\lambda=H\cdot \varphi(F).$ We next show that this divisor is nef and big for large $k$ and vanishes precisely on the ray spanned by $[F].$ For an irreducible curve $A$ in $\tilde X,$ we assume firstly that $A\subseteq E.$ Since $\rho(E)=2,$ $A$ is numerically equivalent to $a F+b  G$ where $G$ is a fiber of $p$  and $a,$ $b$ are positive numbers. Then for any positive $k,$ $\mathcal L_k'\cdot A\ge 0$ and  $\mathcal L_k'\cdot A= 0$ if and only if $b=0.$ Now, suppose that $A\nsubseteq E,$ then by the projection formula, $$\mathcal L_k'\cdot A=k \mathcal L\cdot \varphi(A)-H\cdot \varphi(A)-\lambda E\cdot A=k \bar{H}\cdot \beta(\varphi(A))-H\cdot \varphi(A)-\lambda E\cdot A.$$ The intersection number with $\bar{H},$  we denote it by $<\bar{H}, \cdot
>,$ defines a linear form on $N_1(\bar{X})_{\mathbb R}.$ Since $\bar{H}$ is ample, there is a positive lower bound of $<\bar{H}, \cdot
>$ for any compact subset of $\overline{NE(\bar X)}\setminus\{0\}.$ So for large enough $k,$ $\mathcal L_k'\cdot A>0.$ In conclusion, for $k>>0,$ $\mathcal L_k'$ is nef and big and vanishes precisely on the ray spanned by $[F].$ Then $F$ is a $K_{\tilde{X}}$ negative extremal curve, so $X^+$ is projective.\eexample We'd like to thank Will Donovan for pointing out that our Example \ref{Example} is already known by Roland Aburf, see \cite{Ed Segal}.

In conclusion, we have proved the following theorem.
\begin{theorem} A birational map of simple type, if $\dim X\le 5,$
is either a  standard flop, a twisted Mukai flop or a flop as in Example \ref{Example}.\end{theorem}

\begin{thebibliography}{99}\bibitem{De} O.Debarre:
\emph{Higher-Dimensional Algebraic Geometry. Universitext.
Springer-Verlag, New York, 2001.}\bibitem{F}A.Fujiki, S.Nakano: \emph{Supplement to ``On the inverse of monoidal transformation''. Publ.RIMS Kyoto Univ.7(1071), 637-644.}
 \bibitem{BW}B.Fu, C.-L. Wang:\emph{
Motivic and quantum invariance under stratified Mukai flops. (English summary)
J. Differential Geom. 80 (2008), no. 2, 261-280.}
\bibitem{H}A.H$\ddot{o}$ring, C.Novelli: \emph{Mori contraction of maximal length.  Publ. Res. Inst. Math. Sci. 49 (2013), no. 1, 215-228.}
\bibitem{DH}D.Huybrechts:
\emph{Fourier-Mukai transforms in algebraic geometry.
Oxford Mathematical Monographs. The Clarendon Press, Oxford University Press, Oxford, 2006. viii+307 pp.}
    \bibitem{K} A.Kanemitsu: \emph{Fano 5-folds with nef tangent bundles.arxiv.org/abs/1503.04579.}

    \bibitem{L} R.Lazarsfeld: \emph{Some applications of the theory of positive vector bundles. Ann. of Math. 110 (1979)， 593-606.}
    \bibitem{W.C.L}Y.P. Lee, H.-W. Lin,  C.-L. Wang: \emph{
Flops, motives, and invariance of quantum rings. (English summary)
Ann. of Math. (2) 172 (2010), no. 1, 243-290. }
\bibitem{MM1}M.Maruyama: \emph{On classification of ruled surfaces.
Lectures in Mathematics, Department of Mathematics, Kyoto University, 3 Kinokuniya Book-Store Co., Ltd., Tokyo 1970 iv+75 pp.}
\bibitem{MM2}M.Maruyama:  \emph{On a family of algebraic vector bundles. Number theory, algebraic geometry and commutative algebra, in honor of Yasuo Akizuki, pp. 95-146. Kinokuniya, Tokyo, 1973. }
    \bibitem{Milne}James S. Milne: \emph{\'{E}tale cohomology.
Princeton Mathematical Series, 33. Princeton University Press, Princeton, N.J., 1980. xiii+323 pp.}
   \bibitem{O}C. Okonek, M. H. Schneider and H. Spindler: \emph{Vector bundles on complex projective spaces.
Corrected reprint of the 1988 edition. With an appendix by S. I. Gelfand. Modern Birkh\"{a}user Classics. Birkh\"{a}user/Springer Basel AG, Basel, 2011. viii+239 pp.}
    \bibitem{A} G.Occhetta, Jaros{\l}aw A. Wi\'{s}niewski:  \emph{On Euler- Jaczewski sequence and Remmert- Van de Ven problem for toric varieties. 《Mathematische Zeitschrift》, 2001, 241(1):35-44.}
    \bibitem{Sato} E.SATO: \emph{Varieties which have two projective space bundle structures. J.Math.Kyoto Univ. 25-3 (1985)445-457.}
    \bibitem{Ed Segal}Ed Segal: \emph{A new 5-fold flop and derived equivalence.
Bull. Lond. Math. Soc. 48 (2016), no. 3, 533-538. }\bibitem{S}M.Szurek; Jaros{\l}aw A. Wi\'{s}niewski:
\emph{Fano bundles over $P^3$ and $Q_3$.
Pacific J. Math. 141 (1990), no. 1, 197-208. }
    \bibitem{W} K.Watanabe: \emph{$\P^1$-bundles admitting another smooth morphism of relative dimension one. J. Algebra 414 (2014), 105-119.}

    \end{thebibliography}
\end{document}